\documentclass[a4paper,10pt]{amsproc}
\usepackage{graphicx}
\usepackage{amsfonts, amsmath, amssymb, graphicx, mathrsfs}
\usepackage[all]{xy}
\usepackage{color}

\newdir{ >}{!/6pt/@{ }*:(1,0.2)@_{>}*:(1,-0.2)@^{>}}

\newtheorem{theorem}{Theorem}[section]
\newtheorem{lemma}[theorem]{Lemma}

\newtheorem{definition}[theorem]{Definition}

\newtheorem{remark}[theorem]{Remark}

\newtheorem{counter example}[theorem]{Counter Example}

\definecolor{dgreen}{rgb}{0,.5,.0}
\definecolor{purple}{rgb}{.5,0,.5}
\definecolor{lgreen}{rgb}{.5,1,.5}
\definecolor{dyellow}{rgb}{.5,.5,0}
\definecolor{vdyellow}{rgb}{.25,.25,0}
\definecolor{dred}{rgb}{.5,0,0}
\definecolor{dblue}{rgb}{0,0,.5}
\definecolor{dpurple}{rgb}{.5,0,.5}

\newcommand{\C}{\mathbb{C}}

\begin{document}
\title[Complex-valued continuous functions]{Intermediate rings of complex-valued continuous functions}

\author[A. Acharyya]{Amrita Acharyya}
\address{Department of Mathematics and Statistics,
	University of Toledo, Main Campus,
	Toledo, OH 43606-3390}
\email{amrita.acharyya@utoledo.edu}

\author[S. K. Acharyya]{Sudip Kumar Acharyya}
\address{Department of Pure Mathematics, University of Calcutta, 35, Ballygunge Circular Road, Kolkata 700019, West Bengal, India}
\email{sdpacharyya@gmail.com}

\author[S. bag]{Sagarmoy Bag}
\address{Department of Pure Mathematics, University of Calcutta, 35, Ballygunge Circular Road, Kolkata 700019, West Bengal, India}
\thanks{The author thanks the NBHM, Mumbai-400 001, India, for financial support}
\email{sagarmoy.bag01@gmail.com}

\author[J. Sack]{Joshua Sack} 
\address{Department of Mathematics and Statistics, California State University Long Beach, 1250 Bellflower Blvd, Long Beach, CA 90840, USA}
\email{joshua.sack@csulb.edu}


\thanks {}

\maketitle

\section*{Abstract}
Let $\Sigma (X,\mathbb{C})$ denote the collection of all the rings between $C^*(X,\mathbb{C})$ and $C(X,\mathbb{C})$. We show that there is a natural correlation between the absolutely convex ideals/ prime ideals/maximal ideals/$z$-ideals/$z^\circ$-ideals in the rings $P(X,\mathbb{C})$ in $\Sigma(X,\mathbb{C})$ and in their real-valued counterparts $P(X,\mathbb{C})\cap C(X)$. 
It is shown that the structure space of any such $P(X,\mathbb{C})$  is $\beta X$. We show that for any maximal ideal $M$ in $C(X,\mathbb{C}), C(X,\mathbb{C})/M$ is an algebraically closed field. We give a necessary and sufficient condition for the ideal $C_{\mathcal{P}}(X,\mathbb{C})$ of $C(X,\mathbb{C})$ to be a prime ideal, and we examine a few special cases thereafter. 

\section{Introduction}
In what follows, $X$ stands for a completely regular Hausdorff topological space and $C(X,\mathbb{C})$ denotes the ring of all complex-valued continuous functions on $X$. $C^*(X,\mathbb{C})$ is the subring of $C(X,\mathbb{C})$ containing those functions which are bounded over $X$. As usual $C(X)$ designates the ring of all real-valued continuous functions on $X$ and $C^*(X)$ consists of those functions in $C(X)$ which are bounded over $X$. An intermediate ring of real-valued continuous functions on $X$ is a ring that lies between $C^*(X)$ and $C(X)$. Let $\Sigma (X)$ be the aggregate of all such rings.
Likewise an intermediate ring of complex-valued continuous functions on $X$ is a ring lying between $C^*(X,\mathbb{C})$ and $C(X,\mathbb{C})$. Let $\Sigma (X,\mathbb{C})$ be the family of all such intermediate rings.
It turns out that each member $P(X,\mathbb{C})$ of $\Sigma (X,\mathbb{C})$ is absolutely convex in the sense that $\lvert f\rvert \leq \lvert g\rvert, g\in P(X,\mathbb{C}), f\in C(X,\mathbb{C})$ implies $f\in P(X,\mathbb{C})$. It follows that each such $P(X,\mathbb{C})$ is \emph{conjugate-closed} in the sense that if whenever $f+ ig\in P(X,\mathbb{C})$ where $f,g\in C(X)$, then $f-ig\in P(X,\mathbb{C})$.
It is realised that there is a natural correlation between the prime ideals/ maximal ideals/ $z$-ideals/ $z^\circ$-ideals in the rings $P(X,\mathbb{C})$ and the prime ideals / maximal ideals/ $z$-ideals in the ring $P(X,\mathbb{C})\cap C(X)$. In the second and third sections of this article, we examine these correlations in some details. Incidentally an interconnection between prime ideals in the two rings $C(X,\mathbb{C})$ and $C(X)$ is already observed long time back in Corollary 1.2\cite{Alling}. As a follow up of our investigations on the ideals in these two rings, we establish that the structure spaces of the two rings $P(X,\mathbb{C})$ and $P(X,\mathbb{C})\cap C(X)$ are homeomorphic.
The structure space of a commutative ring $R$ with unity stands for the set of all maximal ideals of $R$ equipped with the well-known hull-kernel topology. It was established in \cite{P} and \cite{RW1987}, independently that the structure space of all the intermediate rings of real-valued continuous functions on $X$ are one and the same viz the Stone-\v{C}ech compactification $\beta X$ of $X$. It follows therefore that the structure space of each intermediate ring of complex-valued continuous functions on $X$ is also $\beta X$.
This is one of the main technical results in our article. We like to mention in this context that a special case of this result telling that the structure space of $C(X,\mathbb{C})$ is $\beta X$ is quite well known, see \cite{MR}. We call a ring $A(X,\mathbb{C})$ in the family $\Sigma (X,\mathbb{C})$ a $C$-type ring if it is isomorphic to a ring of the form $C(Y,\mathbb{C})$ for a Tychonoff space $Y$. We establish that if $I$ is any ideal of $C(X,\mathbb{C})$, then the linear sum $C^*(X,\mathbb{C})+I$ is a $C$-type ring.
This is the complex analogue of the corresponding result in the intermediate rings of real-valued continuous functions on $X$ as proved in \cite{DO}. We further realise that these are the only $C$-type intermediate rings in the family $\Sigma (X,\mathbb{C})$ when and only when $X$ is pseudocompact i.e. $C(X,\mathbb{C})=C^*(X,\mathbb{C})$.

It is well-known that if $M$ is a maximal ideal in $C(X)$, then the residue class field $C(X)/M$ is real closed in the sense that every positive element in this field is a square and each odd degree polynomial over this field has a root in the same field \cite[Theorem 13.4]{GJ}.
The complex analogue of this result as we realise is that for a maximal ideal $M$ in $C(X,\mathbb{C}), C(X,\mathbb{C})/M$ is an algebraically closed field and furthermore this field is the algebraic closure of $C(X)/M\cap C(X)$.

In section 4 of this article, we deal with a few special problems originating from an ideal $\mathcal{P}$ of closed sets in $X$ and a certain class of ideals in the ring $C(X,\mathbb{C})$.
A family $\mathcal{P}$ of closed sets in $X$ is called an ideal of closed sets in $X$ if for any two sets $A, B$ in $\mathcal{P}, A\cup B\in \mathcal{P}$ and for any closed set $C$ contained in $A, C$ is also a member of $\mathcal{P}$. We let $C_\mathcal{P} (X,\mathbb{C})$ be the set of all those functions $f$ in $C(X,\mathbb{C})$ whose support $cl_X (X\setminus Z(f))$ is a member of $\mathcal{P}$; here $Z(f)=\{x\in X:f(x)=0\}$ is the zero set of $f$ in $X$.
We determine a necessary and sufficient condition for $C_\mathcal{P} (X,\mathbb{C})$ to become a prime ideal in the ring $C(X,\mathbb{C})$ and examine a few special cases corresponding to some specific choices of the ideal $\mathcal{P}$.
The ring $C_\infty (X,\mathbb{C})=\{f\in C(X,\mathbb{C}): f \textrm{ vanishes at infinity in the sense that for each }n\in \mathbb{N}, \{x\in X: \lvert f(x)\rvert\geq \frac{1}{n}\}\textrm{ is compact}\}$ is an ideal of $C^*(X,\mathbb{C})$ but not necessarily an ideal of $C(X,\mathbb{C})$. On the assumption that $X$ is locally compact, we determine a necessary and sufficient condition for $C_\infty (X,\mathbb{C})$ to become an ideal of $C(X,\mathbb{C})$.

The fifth section of this article is  devoted to finding out the estimates of a few standard parameters concerning zero divisor graphs of a few rings of complex-valued continuous functions on $X$. Thus for instance we have checked that if $\Gamma (A(X,\mathbb{C}))$ is the zero divisor graph of an intermediate ring $A(X,\mathbb{C})$ belonging to the family $\Sigma (X,\mathbb{C})$, then each cycle of this graph has length 3, 4 or 6 and each edge is an edge of a cycle with length 3 or 4. These are the complex analogous of the corresponding results in the zero divisor graph of $C(X)$ as obtained in \cite{AM}.

\section{Ideals in intermediate rings}

Notation: For any subset $A(X)$ of $C(X)$ such that $0\in A(X)$, we set $[A(X)]_c=\{f+ig: f,g\in A(X)\}$ and call it the \emph{extension} of $A(X)$. Then it is easy to see that $[A(X)]_c\cap C(X)=A(X)=[A(X)]_c\cap A(X)$. From now on, unless otherwise stated, we assume that $A(X)$ is an intermediate ring of real-valued continuous functions on $X$, i.e. $A(X)$ is a member of the family $\Sigma (X)$. It follows at once that $[A(X)]_c$ is an intermediate ring of complex-valued continuous functions and it is not hard to verify that $[A(X)]_c$ is the smallest intermediate ring in $\Sigma(X,\C)$ which contains $A(X)$ and the constant function $i$. Furthermore $[A(X)]_c$ is conjugate-closed meaning that if $f+ig\in [A(X)]_c$ with $f,g\in A(X)$, then $f-ig\in [A(X)]_c$.
The following result tells that intermediate rings in the family $\Sigma (X,\mathbb{C})$ are the extensions of intermediate rings in $\Sigma (X)$.

\begin{theorem}\label{thm:conjugateclosedcrit}
A ring $P(X,\mathbb{C})$ of complex valued continuous functions on $X$ is a member of $\Sigma (X,\mathbb{C})$ if and only if there exists a ring $A(X)$ in the family $\Sigma (X)$ such that $P(X,\mathbb{C})=[A(X)]_c$.
\end{theorem}

\begin{proof}
Assume that $P(X,\mathbb{C})\in \Sigma (X,\mathbb{C})$ and let $A(X)=P(X,\mathbb{C})\cap C(X)$.
Then it is clear that $A(X)\in \Sigma (X)$ and $[A(X)]_c\subseteq P(X,\mathbb{C})$.

To prove the reverse containment, let $f+ig\in P(X,\C)$.
Here $f,g\in C(X)$.
Since $P(X,\C)$ is conjugate closed, $f-ig\in P(X,\C)$, and hence $2f$ and $2ig$ both belong to $P(X,\C)$.
Since constant functions are bounded and hence in $P(X,\C)$,
both the constant functions $\frac{1}{2}$ and $\frac{1}{2i}$ are in $P(X,\C)$.
It follows that both $f$ and $g$ are in $P(X,\C)\cap C(X)$, and hence in $A(X)$.
Consequently, $f+ig\in [A(X)]_c$.
Thus, $P(X,\C)\subseteq [A(X)]_c$.
\end{proof}

The following facts involving convex sets will be useful.  A subset $S$ of $C(X)$ is called \emph{absolutely convex} if whenever $\lvert f\rvert\leq \lvert g\rvert$ with $g\in S$ and $f\in C(X)$, then $f\in S$.

\begin{theorem}\label{rw1997result}
Let $A(X)\in \Sigma (X)$. Then

\begin{enumerate}
\renewcommand{\theenumi}{(\alph{enumi})}
\renewcommand{\labelenumi}{\theenumi}

\item \label{rw1997resultItema} $A(X)$ is an absolutely convex subring of $C(X)$ (in the sense that if $\lvert f\rvert\leq \lvert g\rvert$ with $g\in A(X)$ and $f\in C(X)$, then $f\in A(X)$) (\cite[Proposition 3.3]{DO}).

\item \label{rw1997resultpartb} A prime ideal $P$ in $A(X)$ is an absolutely convex subset of $A(X)$ (\cite[Theorem 2.5]{BW}).
\end{enumerate}
\end{theorem}

The following convenient formula for $[A(X)]_c$ with $A(X)\in \Sigma (X)$ will often be helpful to us.

\begin{theorem}
For any $A(X)\in \Sigma (X)$, $[A(X)]_c=\{h\in C(X,\mathbb{C}): \lvert h\rvert\in A(X)\}$.
\end{theorem}

\begin{proof}
First assume that $h=f+ig\in [A(X)]_c$ with $f,g\in A(X)$. Then $\lvert h\rvert\leq \lvert f\rvert +\lvert g\rvert$.
This implies, in view of Theorem~\ref{rw1997result}\ref{rw1997resultItema}, that $h\in A(X)$ and also $\lvert h\rvert \in A(X)$. Conversely, let $h=f+ig\in C(X,\mathbb{C})$ with $f,g\in C(X)$, be such that $\lvert h\rvert\in A(X)$. This means that $(f^2 +g^2)^{\frac{1}{2}}\in A(X)$. Since $\lvert f\rvert\leq (f^2+g^2)^{\frac{1}{2}}$, this implies in view of Theorem~\ref{rw1997result}\ref{rw1997resultItema} that $f\in A(X)$. Analogously $g\in A(X)$. Thus $h\in [A(X)]_c$.
\end{proof}

The proof of the following proposition is routine.

\begin{theorem}\label{thm:IdealExtensionTightness}
If $I$ is an ideal in $A(X)\in \Sigma (X)$, then $I_c$ is the smallest ideal in $[A(X)]_c$ containing $I$.
Furthermore $I_c\cap A(X)=I=I_c\cap C(X)$.
\end{theorem}

The following result is a straight forward consequence of Theorem~\ref{thm:IdealExtensionTightness}.

\begin{theorem}\label{thm:idealorderandstrictness}
If $I$ and $J$ are ideals in $A(X)\in \Sigma (X)$, then $I\subseteq J$ if and only if $I_c\subseteq J_c$. Also $I\subsetneq J$ when and only when $I_c\subsetneq J_c$.
\end{theorem}

We have the following convenient formula for $I_c$ when $I$ is an absolutely convex ideal of $A(X)$.

\begin{theorem}\label{thm:absconv}
If $I$ is an absolutely convex ideal of $A(X)$ (in particular if $I$ is a prime ideal or a maximal ideal of $A(X)$), then $I_c=\{h\in [A(X)]_c: \lvert h\rvert\in I\}$.
\end{theorem}

\begin{proof}
Let $h=f+ig\in I_c$.
Then $f,g\in I$.
Since $\lvert h\rvert\leq \lvert f\rvert +\lvert g\rvert$, the absolute convexity of $I$ implies that $\lvert h\rvert\in I$.
Conversely, let $h=f+ig\in [A(X)]_c$ be such that $\lvert h\rvert\in I$.
Here $f,g\in A(X)$.
Since $\lvert f\rvert\leq (f^2+g^2)^{\frac{1}{2}}=\lvert h\rvert$, it follows from the absolute convexity of $I$ that $f\in I$. Analogously $g\in I$. Hence $h\in I_c$.
\end{proof}

The above theorem prompts us to define the notion of an absolutely convex ideal in $P(X,\mathbb{C})\in \Sigma (X,\mathbb{C})$ as follows:

\begin{definition}
An ideal $J$ in $P(X,\mathbb{C})$ in $\Sigma (X,\mathbb{C})$ is called absolutely convex if for $g,h$ in $C(X,\mathbb{C})$ with $\lvert g\rvert\leq \lvert  h\rvert$ and $h\in J$, it follows that $g\in J$.
\end{definition}

The first part of the following proposition is immediate, while the second part follows from Theorem~\ref{thm:conjugateclosedcrit} and Theorem~\ref{thm:absconv}.

\begin{theorem}\label{thm:convexextrest}
Let $P(X,\mathbb{C})\in \Sigma (X,\mathbb{C})$.

\begin{enumerate}
\renewcommand{\theenumi}{(\roman{enumi})}
\renewcommand{\labelenumi}{\theenumi}
\item If $J$ is an absolutely convex ideal of $P(X,\mathbb{C})$, then $J\cap C(X)$ is an absolutely convex ideal of the intermediate ring $P(X,\mathbb{C})\cap C(X)\in \Sigma (X)$.

\item \label{thm:convexextrest:item2} An ideal $I$ in $P(X,\mathbb{C})\cap C(X)$ is absolutely convex in this ring if and only if $I_c$ is an absolutely convex ideal of $P(X,\mathbb{C})$.

\item \label{thm:convexextrest:item3} If $J$ is an absolutely convex ideal of $P(X,\mathbb{C})$, then $J=[J\cap C(X)]_c$.

\end{enumerate}
\end{theorem}

\begin{proof}
\ref{thm:convexextrest:item3} It is trivial that $[J\cap C(X)]_c\subseteq J$. To prove the reverse implication relation let $h=f+ig\in J$, with $f,g\in C(X)$. The absolute convexity of $J$ implies that $\lvert h\rvert\in J$. Consequently $\lvert h\rvert\in J\cap C(X)$. But since $\lvert f\rvert\leq (f^2 +g^2)^{\frac{1}{2}}=\lvert h \rvert$, it follows again due to the absolute convexity of $P(X,\mathbb{C})$ as a subring of $C(X,\mathbb{C})$ that $f\in P(X,\mathbb{C})$. We further use absolute convexity of $J$ in $P(X,\mathbb{C})$ to assert that $f\in J$. Analogously $g\in J$. Thus $h=f+ig\in [J\cap C(X)]_c$. Therefore $J\subseteq [J\cap C(X)]_c$.
\end{proof}

\begin{remark}
For any $P(X,\mathbb{C})\in \Sigma (X,\mathbb{C})$, the assignment $I\mapsto I_c$ provides a one-to-one correspondence between the absolute convex ideals of $P(X,\mathbb{C})\cap C(X)$ and those of $P(X,\mathbb{C})$.
\end{remark}

The following theorem entails a one-to-one correspondence between the prime ideals of $P(X,\mathbb{C})$ and those of $P(X,\mathbb{C})\cap C(X)$.

\begin{theorem}\label{thm:primeidealconditionbf}
Let $P(X,\mathbb{C})$ be member of $\Sigma (X,\mathbb{C})$.
An ideal $J$ of $P(X,\mathbb{C})$ is prime if and only if there exists a prime ideal $Q$ in $P(X,\mathbb{C})\cap C(X)$ such that $J=Q_c$.
\end{theorem}

\begin{proof}
Let $J$ be a prime ideal in $P(X,\mathbb{C})$ and let $Q=J\cap C(X)$ and $A(X) = P(X,\mathbb{C})\cap C(X)$.
Then $Q$ is a prime ideal in the ring $A(X)$.
It is easy to see that $Q_c\subseteq J$.
To prove the reverse containment, let $h=f+ig\in J$, where $f,g\in P(X,\mathbb{C})$.
Note that $P(X,\mathbb{C})=[A(X)]_c$ by Theorem~\ref{thm:conjugateclosedcrit}.
Hence $f,g\in A(X)$ and therefore $f-ig\in P(X,\mathbb{C})$.
As $J$ is an ideal of $P(X,\mathbb{C})$, it follows that $(f+ig)(f-ig)\in J$ i.e, $f^2+g^2\in J\cap C(X)=Q$. Since $Q$ is a prime ideal in $A(X)$, we can apply Theorem~\ref{rw1997result}\ref{rw1997resultpartb}, yielding $f^2\in Q$ and hence $f\in Q$.
Analogously $g\in Q$. Thus $h\in Q_c$. Therefore $J\subseteq Q_c$.

To prove the converse of this theorem, let $Q$ be a prime ideal in $A(X)$.
It follows from Theorem~\ref{thm:absconv} that $Q_c =\{h\in P(X,\mathbb{C}): \lvert h\rvert\in Q\}$ and therefore $Q_c$ is a prime ideal in $P(X,\mathbb{C})$.
\end{proof}

\begin{remark}\label{rmk:allprimeidealsinPXC}
For any $P(X,\mathbb{C})\in \Sigma (X,\mathbb{C})$, the collection of all prime ideals in $P(X,\mathbb{C})$ is precisely $\{Q_c :Q$ is a prime ideal in $P(X,\mathbb{C})\cap C(X)\}$.
\end{remark}

\begin{remark}\label{rmk:allminprimeidealsinPXC}
The collection of all minimal prime ideals in $P(X,\mathbb{C})$ is precisely $\{Q_c: Q$ is a minimal prime ideal in $P(X,\,\mathbb{C})\cap C(X)\}$. [This follows from Remark~\ref{rmk:allprimeidealsinPXC} and Theorem~\ref{thm:idealorderandstrictness}].
\end{remark}

\begin{theorem}\label{thm:allmaxidealsinPXC}
For any $P(X,\mathbb{C})\in \Sigma (X,\mathbb{C})$,
the collection of all maximal ideals in $P(X,\mathbb{C})$ is $\{ M_c: M$ is a maximal ideal of $P(X,\mathbb{C})\cap C(X)\}$.
\end{theorem}

\begin{proof}
Let $M$ be a maximal ideal in $P(X,\mathbb{C})\cap C(X)=A(X)$. Then by Theorem~\ref{thm:primeidealconditionbf}, $M_c$ is a prime ideal in $P(X,\mathbb{C})$. Suppose that $M_c$ is not a maximal ideal in $P(X,\mathbb{C})$, then there exists a prime ideal $T$ in $P(X,\mathbb{C})$ such that $M_c\subsetneq T$. By remark 2.11, there exists a prime ideal $P$ in $A(X)$ such that $J=P_c$. So $M_c\subsetneq P_c$. This implies in view of Theorem 2.5 that $M\subsetneq P$, a contradiction to the maximality of $M$ in $A(X)$.

Conversely, let $J$ be a maximal ideal of $P(X,\mathbb{C})$.
In particular $J$ is a prime ideal in this ring.
By Remark~\ref{rmk:allprimeidealsinPXC}, $J=Q_c$ for some prime ideal $Q$ in $A(X)$. We claim that $Q$ is a maximal ideal in $A(X)$.
Suppose not; then $Q\subsetneq K$ for some proper ideal $K$ in $A(X)$.
Then by Theorem~\ref{thm:idealorderandstrictness}, $Q_c\subsetneq K_c$ and $K_c$ a proper ideal in $P(X,\mathbb{C})$; this contradicts the maximality of $J = Q_c$.
\end{proof}

We next prove analogoues of Remark~\ref{rmk:allprimeidealsinPXC} and Theorem~\ref{thm:allmaxidealsinPXC} for two important classes of ideals viz $z$-ideals and $z^\circ$-ideals in $P(X,\mathbb{C})\in \Sigma (X,\mathbb{C})$.
These ideals are defined as follows.

\begin{definition}
Let $R$ be a commutative ring with unity. For each $a\in R$, let $M_a$ (respectively $P_a$) stand for the intersection of all maximal ideals (respectively all minimal prime ideals) which contain $a$. An ideal $I$ in $R$ is called a $z$-ideal (respectively $z^\circ$-ideal) if for each $a\in I, M_a\subseteq I$ (respectively $P_a\subseteq I$).
\end{definition}
This notion of $z$-ideals is consistent with the notion of $z$-ideal in $C(X)$ (see \cite[4A5]{GJ}). Since each prime ideal in an intermediate ring $A(X)\in \Sigma (X)$ is absolutely convex (Theorem~\ref{rw1997result}\ref{rw1997resultpartb}), it follows from Theorem~\ref{thm:convexextrest}\ref{thm:convexextrest:item2} and Remark~\ref{rmk:allprimeidealsinPXC} that each prime ideal in  $P(X,\mathbb{C})\in \Sigma (X,\mathbb{C})$ is absolutely convex. In particular each maximal ideal is absolutely convex. Now if $I$ is a $z$-ideal in $P(X,\mathbb{C})\in \Sigma (X,\mathbb{C})$ and $\lvert f\rvert\leq \lvert g\rvert, g\in I, f\in P(X,\mathbb{C})$, then $M_g\subseteq I$. Let $M$ be a maximal ideal in $P(X,\mathbb{C})$ containing $g$. It follows due to the absolute convexity of $M$ that $f\in M$. Therefore $f\in M_g\subset I$. Thus each $z$-ideal in $P(X,\mathbb{C})$ is absolutely convex. Analogously it can be proved that each $z^\circ$-ideal in $P(X,\mathbb{C})$ is absolutely convex.

The following subsidiary result can be proved using routine arguments.

\begin{lemma}\label{thm:extensiondistributeintersection}
For any family $\{I_\alpha :\alpha\in \Lambda\}$ of ideals in an intermediate ring $A(X)\in \Sigma (X)$, $(\bigcap_{\alpha\in \Lambda } I_\alpha )_c=\bigcap_{\alpha\in\Lambda }(I_{\alpha })_c$.
\end{lemma}

\begin{theorem}\label{thm:charzidealsinPXC}
An ideal $J$ in a ring $P(X,\mathbb{C})\in \Sigma (X,\mathbb{C})$ is a $z$-ideal in $P(X,\C)$ if and only if there exists a $z$-ideal $I$ in $P(X,\mathbb{C})\cap C(X)$ such that $J=I_c$.
\end{theorem}

\begin{proof}
First assume that $J$ is a $z$-ideal in $P(X,\mathbb{C})$.
Let $I=J\cap C(X)$.
Since $J$ is absolutely convex, it follows from Theorem~\ref{thm:convexextrest}\ref{thm:convexextrest:item3} that $J=I_c$. We show that $I$ is a $z$-ideal in $P(X,\mathbb{C})\cap C(X)$. Choose $f\in I$. Suppose $\{M_\alpha :\alpha\in \Lambda\}$ is the set of all maximal ideals in the ring $P(X,\mathbb{C})\cap C(X)$ which contain $f$. It follows from Theorem~\ref{thm:allmaxidealsinPXC} that $\{(M_\alpha )_c:\alpha\in \Lambda\}$ is the set of all maximal ideals in $P(X,\mathbb{C})$ containing $f$.
Since $f\in J$ and $J$ is a $z$-ideal in $P(X,\mathbb{C})$, it follows that $\bigcap_{\alpha\in\Lambda}(M_\alpha)_c\subseteq J$. This implies in the view of Lemma~\ref{thm:extensiondistributeintersection} that $(\bigcap_{\alpha\in \Lambda }M_\alpha )_c\cap C(X)\subseteq I$, and hence $f\in \bigcap_{\alpha\in \Lambda }M_\alpha \subseteq I$.
Thus it is proved that $I$ is a $z$-ideal in $P(X,\mathbb{C})\cap C(X)$.

Conversely, let $I$ be a $z$-ideal in the ring $P(X,\mathbb{C})\cap C(X)$.
We shall prove that $I_c$ is a $z$-ideal in $P(X,\mathbb{C})$.
We recall from Theorem~\ref{thm:conjugateclosedcrit} that $[P(X,\mathbb{C})\cap C(X)]_c=P(X,\mathbb{C})$. Choose $f$ from $I_c$. From Theorem~\ref{thm:absconv}, it follows that (taking care of the fact that each $z$-ideal in $P(X,\mathbb{C})$ is absolutely convex) $\lvert f\rvert\in I$. Let $\{N_\beta :\beta\in \Lambda^*\}$ be the set of all maximal ideals in $P(X,\mathbb{C})\cap C(X)$ which contain the function $\lvert f\rvert$. The hypothesis that $I$ is a $z$-ideal in $P(X,\mathbb{C})\cap C(X)$ therefore implies that $\bigcap_{\beta \in\Lambda^*}N_\beta\subseteq I$.
This further implies in view of Lemma~\ref{thm:extensiondistributeintersection} that $\bigcap_{\beta\in\Lambda^*}(N_\beta )_c\subseteq I_c$. Again it follows from Theorem~\ref{thm:absconv} that, for any maximal ideal $M$ in $P(X,\mathbb{C})\cap C(X)$ and any $g\in P(X,\mathbb{C})$, $g\in M_c$ if and only if $\lvert g\rvert\in M$.
Thus for any $\beta\in \Lambda^*, \lvert f\rvert\in N_\beta$ if and only if $f\in (N_\beta)_c$.
This means that $\{(N_{\beta})_c\}_{\beta\in\Lambda^*}$ is the collection of maximal ideals in $P(X,\mathbb{C})$ which contain $f$, and we have already observed that $f\in \cap_{\beta\in\Lambda^*}(N_\beta)_c\subseteq I_c$.
Consequently $I_c$ is a $z$-ideal in $P(X,\mathbb{C})$.
\end{proof}

If we use the result embodied in Remark~\ref{rmk:allminprimeidealsinPXC} and take note of the fact that each minimal prime ideal in $P(X,\mathbb{C})$ is absolutely convex and argue as in the proof of Theorem~\ref{thm:charzidealsinPXC}, we get the following proposition:

\begin{theorem}\label{thm:charz0idealsinPXC}
An ideal $J$ in a ring $P(X,\mathbb{C})\in \Sigma (X,\mathbb{C})$ is a $z^\circ$-ideal in $P(X,\C)$ if and only if there exists a $z^\circ$-ideal $I$ in $P(X,\mathbb{C})\cap C(X)$ such that $J=I_c$.
\end{theorem}

An ideal $I$ in $A(X)\in \Sigma (X)$ is called fixed if $\bigcap_{f\in I}Z(f)\neq \emptyset$. The following proposition is a straightforward consequence of Theorem~\ref{thm:IdealExtensionTightness}.

\begin{theorem}\label{thm:charfixedidealsinPXC}
An ideal $J$ in a ring $P(X,\mathbb{C})\in \Sigma (X,\mathbb{C})$ is a fixed ideal in $P(X,\C)$ if and only if $J\cap C(X)$ is a fixed ideal in $P(X,\mathbb{C})\cap C(X)$.
\end{theorem}

We recall that a space $X$ is called an \emph{almost $P$ space} if every non-empty $G_\delta$ subset of $X$ has nonempty interior. These spaces have been characterized via $z$-ideals and $z^\circ$-ideals in the ring $C(X)$ in \cite{AKA1999}. We would like to mention that the same class of spaces have witnessed a very recent characterization in terms of fixed maximal ideals in a given intermediate ring $A(X)\in \Sigma (X)$. We reproduce below these two results to make the paper self-contained.

\begin{theorem}(\cite{AKA1999})\label{thm:AKAalmostPintermsofz0ideals}
$X$ is an almost $P$ space if and only if each maximal ideal in $C(X)$ is a $z^\circ$-ideal if and only if each $z$-ideal in $C(X)$ is a $z^\circ$-ideal.
\end{theorem}

\begin{theorem}(\cite{BAM})\label{thm:AKAalmostPintermsoffixedideals}
Let $A(X)\in \Sigma (X)$ be an intermediate ring of real-valued continuous functions on $X$. Then $X$ is an almost $P$ space if and only if each fixed maximal ideal $M^p_A=\{g\in A(X): g(p)=0\}$ of $A(X)$ is a $z^\circ$-ideal.
\end{theorem}
It is further realised in \cite{BAM} that if $X$ is an almost $P$ space, then the statement of Theorem~\ref{thm:AKAalmostPintermsofz0ideals} cannot be improved by replacing $C(X)$ by an intermediate ring $A(X)$, different from $C(X)$. Indeed it is shown in \cite[Theorem 2.4]{BAM} that if an intermediate ring $A(X)\neq C(X)$, then there exists a maximal ideal in $A(X)$ (which is incidentally also a $z$-ideal in $A(X)$), which is not a $z^\circ$-ideal in $A(X)$.

We record below the complex analogue of the above results.

\begin{theorem}\label{thm:CXCalmostPintermsofz0ideals}
$X$ is an almost $P$ space if and only if each maximal ideal of $C(X,\mathbb{C})$ is a $z^\circ$-ideal if and only if each $z$-ideal in $C(X,\mathbb{C})$ is a $z^\circ$-ideal.
\end{theorem}

\begin{proof}
This follows from combining Theorems~\ref{thm:allmaxidealsinPXC}, ~\ref{thm:charzidealsinPXC},~\ref{thm:charz0idealsinPXC}, and~\ref{thm:AKAalmostPintermsofz0ideals}.
\end{proof}

\begin{theorem}
Let $P(X,\mathbb{C})\in \Sigma (X,\mathbb{C})$. Then $X$ is almost $P$ if and only if each fixed maximal ideal $M^p_P=\{g\in P(X,\mathbb{C}): g(p)=0\}$ of $P(X,\mathbb{C})$ is a $z^\circ$-ideal.
\end{theorem}

\begin{proof}
This follows from combining Theorems~\ref{thm:allmaxidealsinPXC},~\ref{thm:charfixedidealsinPXC}, and~\ref{thm:AKAalmostPintermsoffixedideals}.
\end{proof}

\begin{theorem}\label{thm:CXCamongIntermediateRingsWhenAlmostP}
Let $X$ be an almost $P$ space and let $P(X,\mathbb{C})$ be a member of $\Sigma (X,\mathbb{C})$ such that $P(X,\mathbb{C})\subsetneq C(X,\mathbb{C})$.
Then there exists a maximal ideal in $P(X,\mathbb{C})$, which is not a $z^\circ$-ideal in $P(X,\mathbb{C})$.
\end{theorem}
Thus, within the class of almost $P$-spaces $X$, $C(X,\mathbb{C})$ is characterized amongst all the intermediate rings $P(X,\mathbb{C})$ of $\Sigma (X,\mathbb{C})$ by the requirement that $z$-ideals and $z^\circ$-ideals (equivalently maximal ideals and $z^\circ$-ideals) in $P(X,\mathbb{C})$ are one and the same.

\begin{proof}
This follows from combining Theorems~\ref{thm:allmaxidealsinPXC}, ~\ref{thm:charzidealsinPXC}, and~\ref{thm:charz0idealsinPXC}
of this article together with \cite[Theorem 2.4]{BAM}.
\end{proof}

We recall the classical result that $X$ is a $P$ space if and only if $C(X)$ is a Von-Neumann regular ring meaning that each prime ideal in $C(X)$ is maximal. Incidentally the following fact was rather recently established:

\begin{theorem}(\cite{AB2013},\cite{MSW},\cite{BAM})\label{thm:CXnotregular}
If $A(X)\in \Sigma (X)$ is different from $C(X)$, then $A(X)$ is never a regular ring.
\end{theorem}

Theorems~\ref{thm:primeidealconditionbf},~\ref{thm:allmaxidealsinPXC}, and~\ref{thm:CXnotregular} yield in a straight forward manner the following result:

\begin{theorem}
If $P(X,\mathbb{C})\in \Sigma (X,\mathbb{C})$ is a proper subring of $C(X,\mathbb{C})$, then $P(X,\mathbb{C})$ is not a Von-Neumann regular ring.
\end{theorem}

It is well-known that if $P$ is a non maximal prime ideal in $C(X)$ and $M$ is the unique maximal ideal containing $P$, then the set of all prime ideals in $C(X)$ that lie between $P$ and $M$ makes a Dedekind complete chain containing no fewer than $2^{\aleph_1}$ many members (see \cite[Theorem 14.19]{GJ}). If we use this standard result and combine with Theorems~\ref{thm:idealorderandstrictness},~\ref{thm:primeidealconditionbf}, and~\ref{thm:allmaxidealsinPXC}, we obtain the complex-version of this fact:

\begin{theorem}
Suppose $P$ is a non maximal prime ideal in the ring $C(X,\mathbb{C})$. Then there exists a unique maximal ideal $M$ containing $P$ in this ring. Furthermore, the collection of all prime ideals that are situated between $P$ and $M$ constitutes a Dedeking complete chain containing at least $2^{\alpha_1}$ many members.
\end{theorem}

Thus for all practical purposes (say for example when $X$ is not a $P$ space), $C(X,\mathbb{C})$ is far from being a Noetherian ring. Incidentally we shall decide the Noetherianness condition of $C(X,\mathbb{C})$ by deducing it from a result in Section~\ref{sec:IdealsFormCpCinf}; in particular, we show that $C(X,\mathbb{C})$ is Noetherian if and only if $X$ is a finite set.

\section{Structure spaces of intermediate rings}

We need to recall a few technicalities associated with the hull-kernel topology on the set of all maximal ideals $\mathcal{M}(A)$ of a commutative ring $A$ with unity. If we set for any element $a$ of $A$, $\mathcal{M}(A)_a=\{M\in \mathcal{M}(A): a\in M\}$, then the family $\{\mathcal{M}(A)_a: a\in A\}$ constitutes a base for closed sets of the hull-kernel topology on $\mathcal{M}(A)$. We may write $\mathcal{M}_a$ for $\mathcal{M}(A)_a$ when the context is clear.
The set $\mathcal{M}(A)$ equipped with this hull-kernel topology is called the \emph{structure space} of the ring $A$.

For any subset $\mathcal{M}_\circ$ of $\mathcal{M}(A)$, its closure $\overline{\mathcal{M}_\circ}$ in this topology is given by: $\overline{\mathcal{M}_\circ}=\{M\in \mathcal{M}(A): M\supseteq \bigcap \mathcal{M}_\circ\}$.
For further information on this topology, see \cite[7M]{GJ}.

Following the terminology of \cite{CH}, by a (Hausdorff) compactification of a Tychonoff space $X$ we mean a pair $(\alpha , \alpha X)$, where $\alpha X$ is a compact Hausdorff space and $\alpha : X\rightarrow \alpha X$ a topological embedding with $\alpha (X)$ dense in $\alpha X$.
For simplicity, we often designate such a pair by the notation $\alpha X$. Two compactifications $\alpha X$ and $\gamma X$ of $X$ are called \emph{topologically equivalent} if there exists a homeomorphism $\psi :\alpha X\rightarrow \gamma X$ with the property $\psi\circ \alpha =\gamma$.
A compactification $\alpha X$ of $X$ is said to possess the \emph{extension property} if given a compact Hausdorff space $Y$ and a continuous map $f:X\to Y$, there exists a continuous map $f^\alpha :\alpha X\to Y$ with the property $f^\alpha \circ \alpha =f$. It is well known that the Stone-\v{C}ech compactification $\beta X$ of $X$ or more formally the pair $(e,\beta X)$, where $e$ is the evaluation map on $X$ induced by $C^*(X)$ defined by the formula: $e(x)=(f(x):f\in C^*(X))$ such that $e: X\mapsto \mathbb{R}^{C^*(X)}$ , enjoys the extension property. Furthermore this extension property characterizes $\beta X$ amongst all the compactifications of $X$ in the sense that whenever a compactification $\alpha X$ of $X$ has extension property, it is topologically equivalent to $\beta X$.
For more information on these topic, see \cite[Chapter 1]{CH}.

The structure space $\mathcal{M}(A(X))$ of an arbitrary intermediate ring $A(X)\in \Sigma (X)$ has been proved to be homeomorphic to $\beta X$, independently by the authors in \cite{P} and \cite{RW1987}. Nevertheless we offer yet another independent technique to establish a modified version of the same fact by using the above terminology of \cite{CH}.

\begin{theorem}\label{thm:compactificationconstruction}
Let $\eta_A :X\rightarrow \mathcal{M}(A(X))$ be the map defined by $\eta_A (x)=M_A^x=\{g\in A(X):g(x)=0\}$ (a fixed maximal ideal in $A(X)$). Then the pair $(\eta_A , \mathcal{M}(A(X)))$ is a (Hausdorff) compactification of $X$, which further satisfies the extension property. Hence the pair $(\eta_A ,\mathcal{M}(A(X)))$ is topologically equivalent to the Stone-\v{C}ech compactification $\beta X$ of $X$.
\end{theorem}

\begin{proof}
Since $X$ is Tychonoff, $\eta_A$ is one-to-one. Also $cl_{\mathcal{M}(A(X))}\eta_A(X)=\{M\in\mathcal{M}(A(X)): M\supseteq \bigcap_{x\in X}M^x_A\}=\{M\in \mathcal{M}(A(X)):M\supseteq \{0\}\}=\mathcal{M}(A(X))$. It follows from a result proved in \cite{RW1997} that $\mathcal{M}(A(X))$ is a compact Hausdorff space and $\eta_A$ is an embedding. Thus $(\eta_A, \mathcal{M}(A(X)))$ is a compactification of $X$. To prove that this compactification of $X$ possesses the extension property we take a compact Hausdorff space $Y$ and a continuous map $f:X\rightarrow Y$. It suffices to define a continuous map $f^{\beta_A}:\mathcal{M}(A(X))\rightarrow Y$ with the property that $f^{\beta_A}\circ \eta_A =f$. Let $M$ be any member of $\mathcal{M}(A(X))$ i.e. $M$ is a maximal ideal of the ring $A(X)$. Define $\hat{M}=\{g\in C(Y): g\circ f\in M\}$. Note that if $g\in C(Y)$ then $g\circ f\in C(X)$. Further note that since $Y$ is compact and $g\in C(Y)$, $g$ is bounded i.e. $g(Y)$ is a bounded subset of $\mathbb{R}$. It follows that $(g\circ f)(X)$ is a bounded subset of $\mathbb{R}$ and hence $g\circ f\in C^*(X)$. Consequently $g\circ f\in A(X)$. Thus the definition of $\hat{M}$ is without any ambiguity. It is easy to see that $\hat{M}$ is an ideal of $C(Y)$. It follows, since $M$ is a maximal ideal and therefore a prime ideal of $A(X)$, that $\hat{M}$ is a prime ideal of $C(Y)$. Since $C(Y)$ is a Gelfand ring, $\hat{M}$ can be extended to a unique maximal ideal $N$ in $C(Y)$. Since $Y$ is compact, $N$ is fixed (see \cite[Theorem 4.11]{GJ}). Thus we can write: $N=N_y=\{g\in C(Y): g(y)=0\}$ for some $y\in Y$. We observe that $y\in \bigcap_{g\in \hat{M}}Z(g)$. Indeed $\bigcap_{g\in \hat{M}}Z(g)=\{y\}$ for if $y_1, y_2\in \bigcap_{g\in \hat{M}}Z(g)$, for $y_1\neq y_2$, then $\hat{M}\subseteq N_{y_1}$ and $\hat{M}\subseteq N_{y_2}$ which is impossible as $N_{y_1}\neq N_{y_2}$ and $C(Y)$ is a Gelfand ring.  We then set $f^{\beta_A}(M)=y$. Note that $\{f^{\beta_A}(M)\}=\bigcap_{g\in\hat{M}}Z(g)$. Thus $f^{\beta_A}:\mathcal{M}(A(X))\rightarrow Y$ is a well defined map. Now choose $x\in X$ and then $g\in \hat{M^x_A}$; then $g\circ f\in M^x_A$, which implies that $(g\circ f)(x)=0$. Consequently $f(x)\in Z(g)$ for each $g\in \hat{M^x_A}$. On the other hand $\{f^{\beta_A}(M^x_A)\}=\bigcap_{g\in \hat{M^x_A}}Z(g)$. This implies that $f^{\beta_A}(M^x_A)=f(x)$; in other words $(f^{\beta_A}\circ \eta_A)(x)=f(x)$ and this relation is true for each $x\in X$. Hence $f^{\beta_A}\circ \eta_A=f$.

Now towards the proof of the continuity of the map $f^{\beta_A}$, choose $M\in \mathcal{M}(A(X))$ and a neighbourhood $W$ of $f^{\beta_A}(M)$ in the space $Y$. In a Tychonoff space, every neighbourhood of a point $x$ contains a zero set neighbourhood of $x$, which contains a co-zero set neighbourhood of $x$. So there exist some $g_1,g_2\in C(Y)$, such that $f^{\beta_A}(M)\in Y\setminus Z(g_1)\subseteq Z(g_2)\subseteq W$. It follows that $g_1g_2=0$ as $Z(g_1)\cup Z(g_2)=Y$ which means that $Z(g_1g_2)=Y$. Furthermore $f^{\beta_A}(M)\notin Z(g_1)$. Since $\{f^{\beta_A}(M)\}=\bigcap_{g\in \hat{M}}Z(g)$, as observed earlier, we then have $g_1\notin \hat{M}$. This means that $g_1\circ f\notin M$. In other words $M\in \mathcal{M}(A(X))\setminus \mathcal{M}_{g_1\circ f}$, which is an open neighbourhood of $M$  in $\mathcal{M}(A(X))$. We shall check that $f^{\beta_A}(\mathcal{M}(A(X))\setminus \mathcal{M}_{g_1\circ f})\subseteq W$ and that settles the continuity of $f^{\beta_A}$ at $M$. Towards that end, choose a maximal ideal $N\in \mathcal{M}(A(X))\setminus \mathcal{M}_{g_1\circ f}$. This means that $N\notin \mathcal{M}_{g_1\circ f}$, i.e.\ $g_1\circ f\notin N$. Thus $g_1\notin \hat{N}$. But as $g_1g_2=0$ and $\hat{N}$ is prime ideal in $C(Y)$, it must be that $g_2\in \hat{N}$. Since $\{f^{\beta_A}(N)\}=\bigcap_{g\in \hat{N}}Z(g)$, it follows that $f^{\beta_A}(N)\in Z(g_2)\subseteq W$.
\end{proof}

To achieve the complex analogue of the above mentioned theorem, we need to prove the following proposition, which is by itself a result of independent interest.

\begin{theorem}\label{thm:homeomorphismstructurespaces}
Let $A(X)\in \Sigma (X)$. Then the map $\psi_A: \mathcal{M}([A(X)]_c)\rightarrow \mathcal{M}(A(X))$ mapping $M\rightarrow M\cap A(X)$ is a homeomorphism from the structure space of $[A(X)]_c$ onto the structure space of $A(X)$.
\end{theorem}

\begin{proof}
That the above map $\psi_A$ is a bijection between the structure spaces of $[A(X)]_c$ and $A(X)$ follows from Theorems~\ref{thm:conjugateclosedcrit},~\ref{thm:IdealExtensionTightness},~\ref{thm:idealorderandstrictness}, and~\ref{thm:allmaxidealsinPXC}.
Recall (same notation as before) that $\mathcal{M}([A(X)]_c)_f$ is the set of maximal ideals in the ring $[A(X)]_c$ containing the function $f\in [A(X)]_c$. A typical basic closed set in the structure space $\mathcal{M}([A(X)]_c)$ is given by $\mathcal{M}([A(X)]_c)_h$ where $h\in [A(X)]_c$. Note that $\mathcal{M}([A(X)]_c)_h=\{J\in \mathcal{M}([A(X)]_c):h\in J\}$. So for $h\in [A(X)]_c$, $J\in \mathcal{M}([A(X)]_c)_h$ if and only if $h\in J$, and this is true in view of Theorem~\ref{thm:absconv} and the absolute convexity of maximal ideals (see Theorem~\ref{rw1997result}\ref{rw1997resultpartb} of the present article) if and only if $\lvert h\rvert\in J\cap A(X)$, and this holds when and only when $J\cap A(X)\in \mathcal{M}(A(X))_{\lvert h\rvert}$, which is a basic closed set in the structure space $\mathcal{M}(A(X))$ of the ring $A(X)$. Thus
\begin{equation}\label{eq:homeomorphismproof}
\psi_A[\mathcal{M}([A(X)]_c)_h]=\mathcal{M}(A(X))_{\lvert h\rvert}
\end{equation}
Therefore $\psi_A$ carries a basic closed set in the domain space onto a basic closed set in the range space. Now for a maximal ideal $N$ in $A(X)$ and a function $g\in A(X), g$ belongs to $N$ if and only if $\lvert g\rvert\in N$, because of the absolutely convexity of a maximal ideal in an intermediate ring. Consequently $\mathcal{M}(A(X))_g=\mathcal{M}(A(X))_{\lvert g\rvert}$ for any $g\in A(X)$. Hence from relation~\eqref{eq:homeomorphismproof}, we get: $\psi_A[\mathcal{M}([A(X)]_c)_g]=\mathcal{M}(A(X))_g$ which implies that $\psi_A^{-1}[\mathcal{M}(A(X))_g]=\mathcal{M}([A(X)]_c)_g$. Thus $\psi_A^{-1}$ carries a basic closed set in the structure space $\mathcal{M}(A(X))$ onto a basic closed in the structure space $\mathcal{M}([A(X)]_c)$.  Altogether $\psi_A$ becomes a homeomorphism.

For any $x\in X$ and $A(X)\in \Sigma (X)$, set $M_{A[C]}^x=\{h\in [A(X)]_c: h(x)=0\}$. It is easy to check by using standard arguments, such as those employed to prove the textbook theorem \cite[Theorem 4.1]{GJ}, that $M^x_{A[C]}$ is a fixed maximal in $[A(X)]_c$ and $M^x_{A[C]}\cap A(X)=M^x_A =\{g\in A(X): g(x)=0\}$. Let $\zeta : X\mapsto \mathcal{M}([A(X)]_c)$ be the map defined by: $\zeta_A (x)=M^x_{A[C]}$. Then we have the following results.
\end{proof}

\begin{theorem}
$(\zeta_A , \mathcal{M}([A(X)]_c))$ is a Hausdorff compactification of $X$. Furthermore $(\psi_A\circ \zeta_A)(x)=\eta_A (x)$ for all $x$ in $X$. Hence $(\zeta_A, \mathcal{M}([A(X)]_c))$ is topologically equivalent to the Hausdorff compactification $(\eta_A, \mathcal{M}(A(X)))$ as considered in Theorem~\ref{thm:compactificationconstruction}. Consequently $(\zeta_A, \mathcal{M}([A(X)]_c))$ turns out to be topologically equivalent to the Stone-\v{C}ech compactification $\beta X$ of $X$.
\end{theorem}

\begin{proof}
Since $\mathcal{M}(A(X)$ is Hausdorff \cite{RW1997}, it follows from Theorem~\ref{thm:homeomorphismstructurespaces} that $\mathcal{M}([A(X)]_c)$ is a Hausdorff space. Now by following closely the arguments made at the very beginning of the proof of Theorem~\ref{thm:compactificationconstruction}, one can easily see that $(\zeta_A,\mathcal{M}([A(X)]_c))$ is a Hausdorff compactification of $X$. The second part of the theorem is already realised in Theorem \ref{thm:homeomorphismstructurespaces}. The third part of the present theorem also follows from Theorem~\ref{thm:homeomorphismstructurespaces}.
\end{proof}

\begin{definition}
An intermediate ring $A(X)\in \Sigma (X)$ is called $C$-type in \cite{DO}, if it is isomorphic to $C(Y)$ for some Tychonoff space $Y$.
\end{definition}

In \cite{DO}, the authors have shown that if $I$ is an ideal of the ring $C(X)$, then the linear sum $C^*(X)+I$ is a $C$-type ring and of course $C^*(X)+I\in \Sigma (X)$. Recently the authors in \cite{ABBR} have realised that these are the only $C$-type intermediate rings of real-valued continuous functions on $X$ if and only if $X$ is pseudocompact. We now show that the complex analogous of all these results are also true. We reproduce the following result established in \cite{DA}, which will be needed for this purpose.

\begin{theorem}\label{thm:CringPastResult}
A ring $A(X)\in \Sigma (X)$ is $C$-type if and only if $A(X)$ is isomorphic to the ring $C(\upsilon_A X)$, where $\upsilon_AX=\{p\in \beta X: f^*(p)\in \mathbb{R}$ for each $f\in A(X)\}$ and $f^* :\beta X\mapsto \mathbb{R}\cup \{\infty \}$ is the Stone extension of the function $f$.
\end{theorem}

We extend the notion of $C$-type ring to rings of complex-valued continuous functions: a ring $P(X,\mathbb{C})\in \Sigma (X,\mathbb{C})$ is a \emph{$C$-type ring} if it is isomorphic to a ring $C(Y)$ for some Tychonoff space $Y$.
The following proposition comes of quite naturally.

\begin{theorem}\label{thm:real2complexCtype}
Suppose $A(X)\in \Sigma (X)$ is a $C$-type intermediate ring of real-valued continuous functions on $X$. Then $[A(X)]_c$ is a $C$-type intermediate rings of complex-valued continuous functions on $X$.
\end{theorem}

\begin{proof}
Since $A(X)$ is a $C$-type intermediate ring by Theorem~\ref{thm:CringPastResult}, there exists an isomorphism $\psi :A(X)\mapsto C(\upsilon_A X)$. Let $\hat{\psi}:[A(X)]_c\mapsto C(\upsilon_AX,\mathbb{C})$ be defined as follows: $\hat{\psi}(f+ig)=\psi (f)+i\psi (g)$, where $f,g\in A(X)$. It is not hard to check that $\hat{\psi}$ is an isomorphism on $[A(X)]_c$ onto $C(\upsilon_AX,\mathbb{C})$.
\end{proof}

\begin{theorem}
Let $I$ be a $z$-ideal in $C(X,\mathbb{C})$. Then $C^*(X,\mathbb{C})+I$ is a $C$-type intermediate ring of complex-valued continuous functions on $X$. Furthermore these are the only $C$-type rings lying between $C^*(X,\mathbb{C})$ and $C(X,\mathbb{C})$ if and only if $X$ is pseudocompact.
\end{theorem}

\begin{proof}
As mentioned above, it is proved in \cite{DO} that for any ideal $J$ in $C(X)$, $C^*(X)+J$ is a $C$-type intermediate ring of real-valued continuous functions on $X$.
In light of this and Theorem~\ref{thm:real2complexCtype}, it is sufficient to prove for the first part of this theorem that $C^*(X,\mathbb{C})+I=[C^*(X)+I\cap C(X)]_c$. Towards proving that, let $f,g\in C^*(X)+I\cap C(X)$.
We can write $g=g_1+g_2$ where $g_1\in C^*(X)$ and $g_2\in I\cap C(X)$. It follows that $ig_1\in C^*(X,\mathbb{C})$ and $ig_2\in I$ and this implies that $i(g_1+g_2)\in C^*(X,\mathbb{C})+I$. Thus $f+ig\in C^*(X)+I$. Hence $[C^*(X)+I\cap C(X)]_c\subseteq C^*(X,\mathbb{C})+I$.
To prove the reverse inclusion relation, let $h_1+h_2\in C^*(X,\mathbb{C})+I$, where $h_1\in C^*(X,\mathbb{C})$ and $h_2\in I$. We can write $h_1=f_1+ig_1, h_2=f_2+ig_2$, where $f_1,f_2, g_1,g_2\in C(X)$.
Since $h_1\in C^*(X,\mathbb{C})$, it follows that $f_1,g_1\in C^*(X)$. Thus $\lvert f_2\rvert\leq \lvert h_2\rvert$ and $h_2\in I$.
This implies, because of the absolute convexity of the $z$-ideal $I$ in $C(X,\mathbb{C})$, that $f_2\in I$.
Analogously $g_2\in I$. It is now clear that $f_1+f_2\in C^*(X)+I\cap C(X)$ and $g_1+g_2\in C^*(X)+I\cap C(X)$. Thus $h_1+h_2=(f_1+f_2)+i(g_1+g_2)\in [C^*(X)+I\cap C(X)]_c$. Hence $C^*(X,\mathbb{C})+I\subseteq [C^*(X)+I\cap C(X)]_c$.

To prove the second part of the theorem, we first observe that if $X$ is pseudocompact, then there is practically nothing to prove. Assume therefore that $X$ is not pseudocompact. Hence by \cite{ABBR}, there exists an $A(X)\in \Sigma (X)$ such that $A(X)$ is a $C$-type ring but $A(X)\neq C^*(X)+J$ for any ideal $J$ in $C(X)$. It follows from Theorem~\ref{thm:real2complexCtype} that $[A(X)]_c$ is a $C$-type intermediate ring of complex-valued continuous functions belonging to the family $\Sigma (X,\mathbb{C})$. We assert that there does not exist any $z$-ideal $I$ in $C(X,\mathbb{C})$ with the relation: $C^*(X,\mathbb{C})+I=[A(X)]_C$ and that finishes the present theorem. Suppose towords a contradiction, there exists a $z$-ideal $I$ in $C(X,\mathbb{C})$ such that $C^*(X,\mathbb{C})+I=[A(X)]_C$. Now from the proof of the first part of this theorem, we have already settled that $C^*(X,\mathbb{C})+I=[C^*(X)+I\cap C(X)]_C$. Consequently $[C^*(X)+I\cap C(X)]_C=[A(X)]_C$ which yields $[C^*(X)+I\cap C(X)]_C\cap C(X)=[A(X)]_C\cap C(X)$, and hence $C^*(X)+I\cap C(X)=A(X)$, a contradiction.
\end{proof}

We shall conclude this section after incorporating a purely algebraic result pertaining to the residue class field of $C(X,\mathbb{C})$ modulo a maximal ideal in the same field.

For each $a=(a_1,a_2,\dotsc,a_n)\in \mathbb{C}^n$ if $\mathcal{P}_1a, \mathcal{P}_2a,\dotsc,\mathcal{P}_na$ are the zeroes of the polynomial $P_a (\lambda )=\lambda^n+a_1 \lambda^{n-1}+\dotsb+a_n$, ordered so that $\lvert \mathcal{P}_1 a\rvert\leq \lvert \mathcal{P}_2a\rvert\leq \dotsb \leq \lvert \mathcal{P}_n a\rvert$, then by following closely the arguments of \cite[13.3(a)]{GJ}, the following result can be obtained.

\begin{theorem}\label{thm:kthRootFunctionContinuous}
For each $k$, the function $\mathcal{P}_k :\mathbb{C}^n\mapsto \mathbb{C}$, described above, is continuous.
\end{theorem}

By employing the main argument of \cite[Theorem 13.4]{GJ}, we obtain the following proposition as a consequence of Theorem~\ref{thm:kthRootFunctionContinuous}.

\begin{theorem}\label{thm:residueClassFieldsAlgebraicallyClosed}
For any maximal ideal $N$ in $C(X,\mathbb{C})$, the residue class field $C(X,\mathbb{C})/N$ is algebraically closed.
\end{theorem}

We recall from Theorem~\ref{thm:allmaxidealsinPXC} that the assignment $M\mapsto M_c$ establishes a one-to-one correspondence between maximal ideals in $C(X)$ and those in $C(X,\mathbb{C})$. Let $\phi :C(X)/M\mapsto C(X,\mathbb{C})/M_c$ be the induced assignment between the corresponding residue class fields, explicitly $\phi (f+M)=f+M_c$ for each $f\in C(X)$. It is easy to check that $\phi$ is a ring homomorphism and is one-to-one because if $f+M_c=g+M_c$ with $f,g\in C(X)$, then $f-g\in M_c\cap C(X)=M$ and hence $f+M=g+M$. Furthermore, if we choose an element $f+ig+M_c$ from $C(X,\mathbb{C})/M_c$, with $f,g\in C(X)$, then one can verify easily that it is a root of the polynomial $\lambda^2-2 (f+M_c)\lambda +(f^2+g^2+M_c)$ over the field $\phi (C(X)/M)$. Identifying $C(X)/M$ with $\phi (C(X)/M)$, and taking note of Theorem~\ref{thm:residueClassFieldsAlgebraicallyClosed} we get the following result.

\begin{theorem}
For any maximal ideal $M$ in $C(X)$, the residue class field $C(X,\mathbb{C})/M_c$ is the algebraic closure of $C(X)/M$.
\end{theorem}

\section{Ideals of the form $C_\mathcal{P}(X,\mathbb{C})$ and $C_{\infty}^\mathcal{P}(X,\mathbb{C})$}
\label{sec:IdealsFormCpCinf}

Let $\mathcal{P}$ be an ideal of closed sets in $X$. We set $C_\mathcal{P}(X,\mathbb{C})=\{f\in C(X,\mathbb{C}): cl_X (X\setminus Z(f))\in \mathcal{P}\}$ and $C_\infty^{\mathcal{P}}(X,\mathbb{C})=\{f\in C(X,\mathbb{C}):$ for each $\epsilon>0$ in $\mathbb{R}, \{x\in X: \lvert f(x)\rvert \geq \epsilon\}\in \mathcal{P}\}$. These are the complex analogous of the rings, $C_\mathcal{P}(X)=\{ f\in C(X): cl_X (X\setminus Z(f))\in \mathcal{P}\}$ and $C_\infty^\mathcal{P}(X)=\{f\in C(X):$ for each $\epsilon >0, \{x\in X:\lvert f(x)\rvert\geq \epsilon\}\in \mathcal{P}\}$ already introduced in \cite{AG} and investigated subsequently in \cite{AS}, \cite{BAM}.
As in the real case, it is easy to check that $C_\mathcal{P}(X,\mathbb{C})$ is a $z$-ideal in $C(X,\mathbb{C})$ with $C_\infty^{\mathcal{P}}(X,\mathbb{C})$ just a subring of $C(X,\mathbb{C})$. Plainly we have: $C_\mathcal{P}(X,\mathbb{C})\cap C(X)=C_\mathcal{P}(X)$ and $C_\infty^{\mathcal{P}}(X,\mathbb{C})\cap C(X)=C_\infty^{\mathcal{P}}(X)$.

The following results needs only routine verifications.

\begin{theorem}\label{thm:CPextension}
For any ideal $\mathcal{P}$ of closed sets in $X, [C_{\mathcal{P}}(X)]_C=\{f+ig: f,g\in C_\mathcal{P}(X)\}=C_\mathcal{P}(X,\mathbb{C})$ and $[C_\infty^\mathcal{P}(X)]_C=C_\infty^\mathcal{P}(X,\mathbb{C})$.
\end{theorem}

\begin{theorem}
\begin{itemize}
\item[a)] If $I$ is an ideal of the ring $C_\mathcal{P}(X)$, then $I_c=\{f+ig:f,g\in I\}$ is an ideal of $C_\mathcal{P}(X,\mathbb{C})$ and $I_C\cap C_\mathcal{P}(X)=I$.
\item[b)] If $I$ is an ideal of the ring $C_\infty^{\mathcal{P}}(X)$, then $I_C$ is an ideal of $C_\infty^\mathcal{P}(X,\mathbb{C})$ and $I_C\cap C_\infty^\mathcal{P}(X)=I$.
\end{itemize}
\end{theorem}

We record below the following consequence of the above theorem.

\begin{theorem}\label{thm:nestedIdeals}
If $I_1\subsetneq I_2\subsetneq ....$ is a strictly ascending sequence of ideals in $C_\mathcal{P}(X)$(respectively $C_\infty^\mathcal{P}(X)$), then ${I_1}_c\subsetneq {I_2}_c\subsetneq\cdots $ becomes a strictly ascending sequence of ideals in $C_\mathcal{P}(X,\mathbb{C})$(respectively $C_\infty^{\mathcal{P}}(X,\mathbb{C})$).
\end{theorem}

The analogous results for a strictly descending sequence of ideals in both the rings $C_\mathcal{P}(X)$ and $C_\infty^\mathcal{P}(X)$ are also valid.

\begin{definition}
A space $X$ is called \emph{locally $\mathcal{P}$} if each point of $X$ has an open neighbourhood $W$ such that $cl_X W\in \mathcal{P}$.
\end{definition}
Observe that if $\mathcal{P}$ is the ideal of all compact sets in $X$, then $X$ is local $\mathcal{P}$ if and only if $X$ is locally compact.

Towards finding a condition for which $C_\mathcal{P}(X,\mathbb{C})$ and $C_\infty^\mathcal{P}(X,\mathbb{C})$ are Noetherian ring/Artinian rings, we reproduce a special version of a fact proved in \cite{ACR}:

\begin{theorem}\label{thm:ACR1.1}
(From \cite[Theorem 1.1]{ACR}) Given an ideal $\mathcal{P}$ of closed sets in $X$, the following statements are equivalent for a locally $\mathcal{P}$ space $X$:
\begin{itemize}
\item[1)] $C_\mathcal{P}(X)$ is a Noetherian ring.
\item[2)] $C_\mathcal{P}(X)$ is an Artinian ring.
\item[3)] $C_\infty^\mathcal{P}(X)$ is a Noetherian ring.
\item[4)] $C_\infty^\mathcal{P}(X)$ is an Artinian ring.
\item[5)] $X$ is finite set.
\end{itemize}
\end{theorem}

We also note the following standard result of Algebra.

\begin{theorem}\label{thm:HungerfordIdealThm}
Let $\{R_1, R_2, ..., R_n\}$ be a finite family of commutative rings with identity. The ideals of the direct product $R_1\times R_2\times \dotsb \times R_n$ are exactly of the form $I_1\times I_2\times \dotsb \times I_n$, where for $k=1,2,\dotsc,n$, $I_k$ is an ideal of $R_k$.
\end{theorem}

Now if $X$ is a finite set, with say $n$ elements, then as it is Tychonoff, it is discrete space.
Furthermore if $X$ is locally $\mathcal{P}$, then clearly $\mathcal{P}$ is the power set of $X$. Consequently $C_\mathcal{P}(X,\mathbb{C})=C_\infty^\mathcal{P}(X,\mathbb{C})=C(X,\mathbb{C})=\mathbb{C}^n$, which is equal to the direct product of $\mathbb{C}$ with itself `$n$' times. Since $\mathbb{C}$ is a field, it has just $2$ ideals, hence by Theorem~\ref{thm:HungerfordIdealThm} there are exactly $2^n$ many ideals in the ring $\mathbb{C}^n$. Hence $C_\mathcal{P}(X,\mathbb{C})$ and $C_\infty^\mathcal{P}(X,\mathbb{C})$ are both Noetherian rings and Artinian rings. On the other hand if $X$ is an infinite space and is locally $\mathcal{P}$ space then it follows from the Theorem~\ref{thm:nestedIdeals}  and Theorem~\ref{thm:ACR1.1} that neither of the two rings $C_\mathcal{P}(X,\mathbb{C})$ and $C_\infty^\mathcal{P}(X,\mathbb{C})$ is either Noetherian or Artinian. This leads to the following proposition as the complex analogue of Theorem~\ref{thm:ACR1.1}.

\begin{theorem}
Given an ideal $\mathcal{P}$ of closed sets in $X$, the following statements are equivalent for a locally $\mathcal{P}$ space $X$:
\begin{itemize}
\item[1)] $C_\mathcal{P}(X,\mathbb{C})$ is a Noetherian ring.
\item[2)] $C_\mathcal{P}(X,\mathbb{C})$ is an Artinian ring.
\item[3)] $C_\infty^\mathcal{P}(X,\mathbb{C})$ is a Noetherian ring.
\item[4)] $C_\infty^\mathcal{P}(X,\mathbb{C})$ is an Artinian ring.
\item[5)] $X$ is finite set.
\end{itemize}
\end{theorem}

A special case of this theorem, choosing $\mathcal{P}$ to be the ideal of all closed sets in $X$ reads: $C(X,\mathbb{C})$ is a Noetherian ring if and only if $X$ is finite set.

The following gives a necessary and sufficient condition for the ideal $C_\mathcal{P}(X,\mathbb{C})$ in $C(X,\mathbb{C})$ to be prime.

\begin{theorem}\label{thm:complexAnalogueACR}
Let $\mathcal{P}$ be an ideal of closed sets in $X$ and suppose $X$ is locally $\mathcal{P}$. Then the following statements are equivalent:
\begin{enumerate}
\renewcommand{\theenumi}{(\arabic{enumi})}
\renewcommand{\labelenumi}{\theenumi}
\item\label{thm:complexAnalogueACRitem1} $C_\mathcal{P}(X,\mathbb{C})$ is a prime ideal in $C(X,\mathbb{C})$.
\item\label{thm:complexAnalogueACRitem2} $C_\mathcal{P}(X)$ is a prime ideal in $C(X)$.
\item\label{thm:complexAnalogueACRitem3} $X\notin \mathcal{P}$ and for any two disjoint co-zero sets in $X$, one has its closure lying in $\mathcal{P}$.
\end{enumerate}
\end{theorem}

\begin{proof}
The equivalence of \ref{thm:complexAnalogueACRitem1} and \ref{thm:complexAnalogueACRitem2} follows from Theorem~\ref{thm:primeidealconditionbf} and Theorem~\ref{thm:CPextension}.
Towards the equivalence \ref{thm:complexAnalogueACRitem2} and \ref{thm:complexAnalogueACRitem3}, assume that $C_\mathcal{P}(X)$ is a prime ideal in $C(X)$. If $X\in \mathcal{P}$, then for each $f\in C(X)$, $cl_X (X\setminus Z(f))\in \mathcal{P}$ meaning that $f\in C_\mathcal{P}(X)$ and hence $C_\mathcal{P}(X)=C(X)$, a contradiction to the assumption that $C_\mathcal{P}(X)$ is a prime ideal and in particular a proper ideal of $C(X)$. Thus $X\notin \mathcal{P}$. Now consider two disjoint co-zero sets $X\setminus Z(f)$ and $X\setminus Z(g)$ in $X$, with $f,g\in C(X)$. It follows that $Z(f)\cup Z(g)=X$, i.e. $fg=0$.  Since $C_\mathcal{P}(X)$ is prime, this implies that $f\in C_\mathcal{P}(X)$ or $g\in C_\mathcal{P}(X)$, i.e. $cl_X (X\setminus Z(f))\in \mathcal{P}$ or $cl_X (X\setminus Z(g))\in \mathcal{P}$. 

Conversely let the statement $\ref{thm:complexAnalogueACRitem3}$ be true. Since a $z$-ideal $I$ in $C(X)$ is prime if and only if for each $f,g\in C(X)$, $fg=0$ implies $f\in I$ or $g\in I$ (see \cite[Theorem 2.9]{GJ}) and since $C_\mathcal{P}(X)$ is a $z$-ideal in $C(X)$, it is sufficient to show that for each $f,g\in C(X)$, if $fg=0$ then $f\in C_\mathcal{P}(X)$ or $g\in C_\mathcal{P}(X)$. Indeed $fg=0$ implies that $X\setminus Z(f)$ and $X\setminus Z(g)$ are disjoint co-zero sets in $X$. Hence by supposition $\ref{thm:complexAnalogueACRitem3}$, either $cl_X (X\setminus Z(f))\mathcal{P}$ or $cl_X (X\setminus Z(g))\in \mathcal{P}$ meaning that $f\in C_\mathcal{P}(X)$ or $g\in C_\mathcal{P}(X)$.
\end{proof}

A special case of Theorem~\ref{thm:complexAnalogueACR}, with $\mathcal{P}$ equal to the ideal of all compact sets in $X$, is proved in \cite{AZ}.
We examine a second special case of Theorem~\ref{thm:complexAnalogueACR}.

A subset $Y$ of $X$ is called a \emph{bounded} subset of $X$ if each $f\in C(X)$ is bounded on $Y$. Let $\beta$ denote the family of all closed bounded subsets of $X$. Then $\beta$ is an ideal of closed sets in $X$. It is plain that a pseudocompact subset of $X$ is bounded but a bounded subset of $X$ may not be pseudocompact. Here is a counterexample: the open interval $(0,1)$ in $\mathbb{R}$ is a bounded subset of $\mathbb{R}$ without being a pseudocompact subset of $\mathbb{R}$. However for a certain class of subsets of $X$, the two notions of boundedness and pseudocompactness coincide. The following well-known proposition substantiates this fact:

\begin{theorem}[Mandelkar \cite{M}]\label{thm:Mandelkar}
A support of $X$, i.e.\ a subset of $X$ of the form $cl_X (X\setminus Z(f))$ for some $f\in C(X)$, is a bounded subset of $X$ if and only if it is a pseudocompact subset of $X$.
\end{theorem}

It is clear that the conclusion of Theorem~\ref{thm:Mandelkar} remains unchanged if we replace $C(X)$ by $C(X,\mathbb{C})$.

Let $C_\psi (X)=\{f\in C(X): f$ has pseudocompact support$\}$ and recall that $C_\beta (X)=\{f\in C(X):f$ has bounded support$\}$. We would like to mention here that in spite of the fact that the closed pseudocompact subsets of a space $X$ might not constitute an ideal of closed sets in $X$ (indeed even a closed subset of a pseudocompact space may not be pseucdocompact as is illustrated by the Tychonoff plank in \cite[8.20]{GJ}: $[0,\omega_1]\times [0,\omega]\setminus \{(\omega_1,\omega)\}$, where $\omega_1$ is the 1st uncountable ordinal and $\omega$ is the first infinite ordinal), it is the case that $C_\psi (X)$ is an ideal of the ring $C(X)$. Indeed it follows directly from Theorem~\ref{thm:Mandelkar} that $C_\psi (X)=C_\beta (X)$.

A Tychonoff space $X$ is called \emph{locally pseudocompact} if each point on $X$ has an open neighbourhood with its closure pseudocompact. On the other hand, $X$ is called \emph{locally bounded} (or \emph{locally $\beta$}) if each point in $X$ has an open neighbourhood with its closure bounded. Since each open neighbourhooad of a point $x$ in a Tychonoff space $X$ contains a co-zero set neighbourhood of $x$, it follows from Theorem~\ref{thm:Mandelkar} that $X$ is locally bounded if and only if $X$ is locally pseudocompact. This combined with Theorem~\ref{thm:primeidealconditionbf} leads to the following special case of Theorem~\ref{thm:complexAnalogueACR}.

\begin{theorem}
Let $X$ be locally pseudocompact.
Then the following statements are equivalent:
\begin{enumerate}
\renewcommand{\theenumi}{(\arabic{enumi})}
\renewcommand{\labelenumi}{\theenumi}
\item $C_\psi (X)$ is a prime ideal of $C(X)$
\item $C_\psi (X,\mathbb{C})=\{f\in C(X,\mathbb{C}): f$ has pseudocompact support$\}$ is a prime ideal of $C(X,\mathbb{C})$
\item $X$ is not pseudocompact and for any two disjoint co-zero sets in $X$, the closure of one of them is pseudocompact.
\end{enumerate}
\end{theorem}

Since for $f\in C(X,\mathbb{C})$, $f\in C_\infty (X,\mathbb{C})$ if and only if $\lvert f\rvert\in C_\infty (X)$, it follows that $C_\infty (X,\mathbb{C})$ is an ideal of $C(X,\mathbb{C})$ if and only if $C_\infty (X)$ is an ideal of $C(X)$. In general however $C_\infty (X)$ need not be an ideal of $C(X)$. If $X$ is assumed to be locally compact, then it is proved in \cite{AG2003} and \cite{ASO} that $C_\infty (X)$ is an ideal of $C(X)$ when and only when $X$ is pseudocompact.
Therefore the following theorem holds.

\begin{theorem}
Let $X$ be locally compact. Then the following three statements are equivalent:

\begin{itemize}
\item[1)] $C_\infty (X,\mathbb{C})$ is an ideal of $C(X,\mathbb{C})$.
\item[2)] $C_\infty (X)$ is an ideal of $C(X)$.
\item[3)] $X$ is pseudocompact.
\end{itemize}
\end{theorem}

\section{Zero divisor graphs of rings in the family $\Sigma (X,\mathbb{C})$}

We fix any intermediate ring $P(X,\mathbb{C})$ in the family $\Sigma (X,\mathbb{C})$. Suppose $\mathcal{G}=\mathcal{G}(P(X,\mathbb{C}))$ designates the graph whose vertices are zero divisors of $P(X,\mathbb{C})$ and there is an edge between vertices $f$ and $g$ if and only if $fg=0$. For any two vertices $f,g$ in $\mathcal{G}$, let $d(f,g)$ be the length of the shortest path between $f$ and $g$ and $\operatorname{Diam} \mathcal{G}=\sup \{ d(f,g): f,g\in \mathcal{G}\}$. Suppose $\operatorname{Gr} \mathcal{G}$ designates the length of the shortest cycle in $\mathcal{G}$, often called the girth of $\mathcal{G}$.  It is easy to check that a function $f\in P(X,\mathbb{C})$ is a zero divisor and hence a vertex of $\mathcal{G}$ by checking that $Int_X Z(f)\neq \emptyset$. This parallels the statement that a vertex $f$ in the zero-divisor graph $\Gamma C(X)$ of $C(X)$ considered in \cite{AM} is a divisor of zero in $C(X)$ if and only if $Int_X Z(f)\neq \emptyset$. We would like to point out in this connection that a close scrutiny into the proof of various results in \cite{AM} reveal that several facts related to the nature of the vertices and the length of the cycles related to $\Gamma C(X)$ have been established in \cite{AM} by employing skillfully the last mentioned simple characterization of divisors of zero in $C(X)$. It is expected that the anlogous facts pertaining to the various parameters of the graph $\mathcal{G} (P(X,\mathbb{C}))=\mathcal{G}$ should also hold. We therefore just record the following results related to the graph $\mathcal{G}$, without any proof.

\begin{theorem}
	Let $f,g$ be vertices of the graph $\mathcal{G}$. Then $d(f,g)=1$ if and only if $Z(f)\cup Z(g)=X$; $d(f,g)=2$ when and only when $Z(f)\cup Z(g)\subsetneq X$ and $Int_X Z(f)\cap Int_X Z(g)\neq \phi$; $d(f,g)=3$ if and only if $Z(f)\cup Z(g)\subsetneq X$ and $Int_X Z(f)\cap Int_X Z(g)=\emptyset$. Consequently on assuming that $X$ contains at least $3$ points, $\operatorname{Diam} \mathcal{G}$ and $\operatorname{Gr} \mathcal{G}$ are both equal to $3$, [Compare with \cite[Corollary 1.3]{AM}].
\end{theorem}

\begin{theorem}
	Each cycle in $\mathcal{G}$ has length $3,4$ or $6$. Furthermore every edge of $\mathcal{G}$ is an edge of a cycle with length $3$ or $4$ [compare with \cite[Corollary 2.3]{AM}].
\end{theorem}

\begin{theorem}
	Suppose $X$ contains at least $2$ points. Then
	
	\begin{itemize}
		\item[a)] Each vertex of $\mathcal{G}$ is a $4$ cycle vertex.
		\item[b)] $\mathcal{G}$ is a trianglulated graph meaning that each vertex of $\mathcal{G}$ is a vertex of a triangle when and only when $X$ is devoid of any isolated point.
		
	\item[c)] $\mathcal{G}$ is a hypertriangulated graph in the sense that each edge of $\mathcal{G}$ is edge of a triangle if and only if $X$ is a connected middle $P$ space [compare with the analogous facts in \cite[Proposition 2.1]{AM}].
	\end{itemize}
\end{theorem}

\bibliographystyle{plain}

\end{document}